\documentclass[12pt]{amsart}
\usepackage{amssymb}
\usepackage{amsmath}
\usepackage{amsthm}
\usepackage[all,arc,curve,color,frame]{xy}

\DeclareMathOperator{\Mat}{Mat}

\newcommand{\A}{{\mathcal A}}
\newcommand{\E}{{\mathcal E}}

\renewcommand{\i}{\mathbf {in}}
\renewcommand{\o}{\mathbf {out}}

\newcommand{\M}{{\mathcal M}}

\newcommand{\n}{{\mathbf n}}

\renewcommand{\O}{{\mathcal O}}
\newcommand{\pp}{{\mathbf p}}
\newcommand{\qq}{{\mathbf q}}

\newcommand{\N}{{\mathbb N}}
\newcommand{\Ncal}{{\mathcal N}}
\newcommand{\C}{\ensuremath{\mathbb{C}}}
\newcommand{\R}{\ensuremath{\mathbb{R}}}

\newcommand{\p}{\partial}

\newcommand{\unit}{{\bf{1}}}

\newcommand{\V}{{\mathcal V}}

\newcommand{\Z}{\ensuremath{\mathbb{Z}}}
\newtheorem{lemma}{Lemma}
\newtheorem{proposition}{Proposition}
\newtheorem{theorem}{Theorem}

\begin{document}

\title[Quantization of endomorphism bundle]
{Deformation quantization with separation of variables of an endomorphism bundle}
\author[Alexander Karabegov]{Alexander Karabegov}
\address[Alexander Karabegov]{Department of Mathematics, Abilene
Christian University, ACU Box 28012, Abilene, TX 79699-8012}
\email{axk02d@acu.edu}

\begin{abstract} 
Given a holomorphic Hermitian vector bundle $E$ and a star-product with separation of variables on a pseudo-K\"ahler manifold, we construct a star product on the sections of the endomorphism bundle of the dual bundle $E^\ast$ which also has the appropriately generalized property of separation of variables. For this star product we prove a generalization of Gammelgaard's graph-theoretic formula.
\end{abstract}
\date{April 21, 2013}
\maketitle

\section{Introduction}

Deformation quantization on a Poisson manifold $(M, \{\cdot,\cdot\})$ is an associative product on the space $C^\infty(M)[[\nu]]$ of $\nu$-formal smooth complex-valued functions given by the formula
\begin{equation}\label{E:star}
    f \ast g = \sum_{r \geq 0} \nu^r C_r (f,g),
\end{equation}
where $C_r$ are bidifferential operators on $M$, $C_0(f,g) = fg$, and 
\[
C_1(f,g) - C_1(g,f) = i\{f,g\}.
\]
The product $\ast$ is called a star product. It is assumed that star products are normalized, i.e., the unit constant function $\unit$ is the unity of a star product, $f \ast \unit = \unit \ast f = f$. Two star products $\ast_1, \ast_2$ on $(M, \{\cdot,\cdot\})$ are called equivalent if there exists a formal differential operator $ T = 1 + \nu T_1 + \ldots$ on $M$ such that $T(f \ast_1 g) = Tf \ast_2 Tg$. Star products on $M$ can be restricted (localized) to any open subset of $M$. A star product on $M$ can be extended to   $C^\infty(M)[\nu^{-1},\nu]]$, the space of formal Laurent series of functions with a finite polar part, $f = \nu^s f_s + \nu^{s+1} f_{s+1} + \ldots$, where $s$ is a possibly negative integer.

In the theory of deformation quantization there are general existence and classification results, specific constructions of star products, and explicit formulas for star products. The problem of existence and classification of star products up to equivalence on an arbitrary Poisson manifold was stated in \cite{BFFLS} and settled in \cite{K} by Kontsevich. In \cite{F1} Fedosov gave a geometric construction of star products from every equivalence class on an arbitrary symplectic manifold. 
There are star products on K\"ahler manifolds with the property of separation of variables which originate in the context of Berezin's quantization, such as the Berezin and Berezin-Toeplitz star products (see \cite{Ber}, \cite{E}, \cite{Sch}, \cite{CMP1}, \cite{BW}). Star products with separation of variables on K\"ahler manifolds are bijectively parameterized by formal K\"ahler forms.

When the existence of a specific star product is established, finding an explicit formula for that product may still be a challenging problem. For almost two decades since \cite{BFFLS}, explicit formulas had been known only for a few examples of invariant star products on homogeneous symplectic spaces such as Weyl, $pq$-, $qp$-, Wick, and anti-Wick star-products on linear symplectic spaces, and star-products on complex projective spaces and Grassmann manifolds (see \cite{BBEW}, \cite{Schirm}, and \cite{AL}). For non-invariant star products there are explicit formulas expressed in terms of directed graphs. The first such formula is the celebrated Kontsevich's formula for a star product on $\R^n$ equipped with an arbitrary Poisson structure (see \cite{K}).
There are several explicit graph-theoretic formulas for star products with separation of variables on K\"ahler manifolds. 
In \cite{RT} Reshetikhin and Takhtajan gave a graph-theoretic formula for a star product on an arbitrary K\"ahler manifold which was based upon a formal interpretation of integral formulas for Berezin's quantization. However, the star product given by their explicit formula is not normalized. Inspired by their work, Gammelgaard gave in \cite{G} an explicit formula for the star product with separation of variables with an arbitrary parameterizing formal K\"ahler form. Gammelgaard's formula specifies directly to the Berezin-Toeplitz star product owing to the explicit description of its parameterizing form from \cite{KSch}. Recently Hao Xu found in \cite{Xu1} a graph-theoretic formula for Berezin's star product and calculated its parameterizing form in \cite{Xu2}.

 It is worth mentioning that so far there are no known explicit formulas for Fedosov's star products.

Deformation quantization of endomorphism bundles was used in the proofs of the index theorem for deformation quantization and its generalizations (see \cite{F2}, \cite{NT}, \cite{Aa}, \cite{PPT}). Matrix-valued symbols and related quantizations were considered in \cite{Fan}, \cite{GE}, \cite{MM}, \cite{AE}.

Star products with separation of variables on the sections of the endomorphism bundle of a holomorphic Hermitian vector bundle on a K\"ahler manifold were constructed in \cite{NW} using Fedosov's approach. In this paper we give an alternative construction of such star products in the spirit of \cite{CMP1} and prove a generalized Gammelgaard's graph-theoretic formula for these star products. To this end we generalize the proof of Gammelgaard's formula from \cite{CMP5}. Our sign conventions differ from those in \cite{G} and \cite{CMP5}.

\section{Deformation quantizations with separation of variables}\label{S:sep}

Let $(M, \{\cdot,\cdot\})$ be a Poisson manifold endowed with a complex structure such that the Poisson tensor on $M$ is of type $(1,1)$ with respect to the complex structure. We call $M$ a K\"ahler-Poisson manifold. In the rest of the paper we denote by $m$ the complex dimension of $M$.  In local holomorphic coordinates $z^k, \bar z^l$ we write the K\"ahler-Poisson tensor as $g^{lk}$, so that
\begin{equation}\label{E:kahlpoiss}
  \{f,g\} = i g^{lk} \left(\frac{\p f}{\p z^k}\frac{\p g}{\p \bar z^l} - \frac{\p g}{\p z^k}\frac{\p f}{\p \bar z^l}\right).
\end{equation}
If the tensor $g^{lk}$ is nondegenerate, its inverse $g_{kl}$ is a pseudo-K\"ahler metric tensor on $M$.

A star product (\ref{E:star}) on a K\"ahler-Poisson manifold $M$ is called a star product with separation of variables if the bidifferential operators $C_r$ differentiate their first argument only in antiholomorphic directions and the second argument only in holomorphic ones. Equivalently, if $f$ and $g$ are local functions on $M$ and $f$ is holomorphic or $g$ is antiholomorphic, then
\begin{equation}\label{E:separ}
           f \ast g = fg.
\end{equation}
Star products with separation of variables originate in the context of Berezin's quantization (see \cite{Ber}).
It was shown in \cite{CMP1} and \cite {BW} that star products with separation of variables exist on arbitrary pseudo-K\"ahler manifolds. In the terminology of \cite{BW}, star products with separation of variables are of anti-Wick type.

It is not yet known whether star products with separation of variables exist on arbitrary K\"ahler-Poisson manifolds. Examples of such star products on K\"ahler-Poisson manifolds with degenerate K\"ahler-Poisson tensors are given in \cite{E},\cite{LTW}, and \cite{LMP}.

Deformation quantizations with separation of variables on a pseudo-K\"ahler manifold $M$ were classified in \cite{CMP1}. Denote by $\omega_{-1}$ the pseudo-K\"ahler form on $M$. Consider a formal form
\[
              \omega = \frac{1}{\nu}\omega_{-1} + \omega_0 + \nu \omega_1 + \ldots,
\]
where $\omega_r, r \geq 0,$ are possibly degenerate closed forms of type $(1,1)$ with respect to the complex structure. The star products with separation of variables on $M$ are bijectively parameterized by such formal forms. The star product with separation of variables $\ast$ parameterized by a formal form $\omega$ is completely characterized by the property that
\begin{equation}\label{E:pphik}
      \frac{\p \Phi}{\p z^k} \ast f = \frac{\p \Phi}{\p z^k} f + \frac{\p f}{\p z^k}, \ 1 \leq k \leq m,
\end{equation} 
for any local potential
\[
     \Phi = \frac{1}{\nu}\Phi_{-1} + \Phi_0 + \ldots
\]
of the form $\omega$, so that $\omega = i \p \bar \p \Phi$. Equivalently, $\ast$ is completely characterized by the property that locally
\begin{equation}\label{E:pphil}
    f \ast \frac{\p \Phi}{\p \bar z^l} = \frac{\p \Phi}{\p \bar z^l} f + \frac{\p f}{\p \bar z^l}.
\end{equation}

\section{Deformation quantization of endomorphism bundles}\label{S:oper}

Our definition of a deformation quantization of an endomorphism bundle is a generalization of the definition from \cite{Aa}.
Given  a Poisson manifold $(M, \{\cdot,\cdot\})$, let $E$ be a vector bundle of rank $d$ and $\ast_0$ be a star product on $M$. Denote by $\A$ the star algebra $(C^\infty(M)[[\nu]],\ast_0)$.

A star  product $\ast$ on $C^\infty(End(E))[[\nu]]$ associated to $\ast_0$ is an associative product given by (\ref{E:star}), where $C_r$ are bidifferential operators on $C^\infty(End(E))$ and $C_0(f,g)=fg$ for $f,g \in C^\infty(End(E))$. We require the resulting algebra of formal sections of $C^\infty(End(E))$ to be locally isomorphic to the matrix algebra $Mat_d(\A)$, where the local isomorphisms are given by formal differential operators.

First we set up terminology and notations. Given a holomorphic Hermitian vector bundle $E$ on a complex manifold $M$, we denote by $E^\ast$ and $\bar E$ the dual and the conjugate bundles, respectively, and by $C^\infty(E)$ the space of global smooth sections of $E$. The Hermitian metric on $E$ defines a global linear bijection $u: C^\infty(\bar E) \to C^\infty(E^\ast)$ whose inverse will be denoted by $\tilde u$. A global vector field $\xi$ of type $(1,0)$ on $M$ lifts to an operator on the sections of any antiholomorphic vector bundle on ~$M$. Similarly, a vector field $\bar \xi$ of type $(0,1)$ lifts to an operator on the sections of any holomorphic vector bundle. The canonical connection $\nabla$ on $E^\ast$ is defined by the formulas $\nabla_\xi = u \xi \tilde u$ and $\nabla_{\bar\xi} = \bar\xi$. The canonical connection on $End(E^\ast)$ denoted also by $\nabla$ is defined on $f \in C^\infty(End(E^\ast))$ as follows,
\[
      \nabla_\xi f = u ( \xi(\tilde u f u))\tilde u \mbox{ and } \nabla_{\bar\xi}f = \bar\xi f.
\]

Let $\ast$ be a star product with separation of variables on a K\"ahler-Poisson manifold $(M, g^{lk})$. Given a holomorphic Hermitian vector bundle $E$ of rank $d$ on $M$, we construct a deformation quantization of the endomorphism bundle of the dual bundle $E^\ast$ associated to $\ast$ and show that it has the property of separation of variables.

We choose an open cover $\{U_\alpha\}$ of $M$ such that $E$ has a holomorphic trivialization over each chart $U_\alpha$ and fix these trivializations, $E|_{U_\alpha} \cong U_\alpha \times \C^d$. The corresponding trivialization of the Hermitian metric on $E|_{U_\alpha}$  is given by an invertible matrix $u_\alpha \in Mat_d(C^\infty(U_\alpha))$ whose inverse will be denoted by $\tilde u_\alpha$. Let $a_{\alpha\beta} \in Mat_d(C^\infty(U_\alpha \cap U_\beta))$ be the holomorphic transition function, $b_{\beta\alpha} = \bar a_{\alpha\beta}$ be its complex conjugate, and $\tilde a_{\beta\alpha} = (a_{\alpha\beta})^{-1}$ and $\tilde b_{\alpha\beta} = (b_{\beta\alpha})^{-1}$ be their respective inverses, so that
\[
      u_\alpha = a_{\alpha\beta} u_\beta b_{\beta\alpha}.
\]

Let $\A = (C^\infty(M)[[\nu]],\ast)$ be the star algebra on $M$ and $\A(U)$ be its localization to an open subset $U \subset M$. We denote the product in the algebra $Mat_d(\A)$ also by $\ast$.

\begin{lemma}\label{L:uinv}
An invertible matrix in the algebra $Mat_d (C^\infty(U))$ is also invertible in the algebra $Mat_d (\A(U))$.
\end{lemma}
This statement  is well known in the scalar case ($d$=1). Its proof easily generalizes to the matrix case.

Denote the inverse matrix to $u_\alpha$ in the algebra $Mat_d (\A(U_\alpha))$  by $v_\alpha$.  In particular, $v_\alpha = \tilde u_\alpha \pmod{\nu}$. Using the trivialization of $E|_{U_\alpha}$, we identify $C^\infty(End(E^\ast|_{U_\alpha}))[[\nu]]$ with $Mat_d(C^\infty(U_\alpha)[[\nu]])$. Consider the mapping
\begin{equation}\label{E:astu}
    f_\alpha \mapsto \psi_\alpha = (f_\alpha u_\alpha) \ast v_\alpha
\end{equation}
from $C^\infty(End(E^\ast|_{U_\alpha}))[[\nu]]$ to $Mat_d (\A(U_\alpha))$. It is a $\nu$-linear isomorphism whose inverse is
\[
      \psi_\alpha \mapsto f_\alpha = (\psi_\alpha \ast u_\alpha) \tilde u_\alpha.
\]
{\it Remark.} We stress that in the local formula (\ref{E:astu}) $f_\alpha, u_\alpha, v_\alpha, \psi_\alpha$ are from $Mat_d(C^\infty(U_\alpha))[[\nu]]$ and the matrix product $f_\alpha u_\alpha$ is pointwise.

We transfer the product $\ast$ from the algebra $Mat_d (\A(U_\alpha))$ to the space $C^\infty(End(E^\ast|_{U_\alpha}))[[\nu]]$ via the isomorphism (\ref{E:astu}) and denote the resulting associative $\nu$-linear product by $\ast_\alpha$. The product of sections $f_\alpha,g_\alpha \in C^\infty(End(E^\ast|_{U_\alpha}))[[\nu]] \cong Mat_d(C^\infty(U)[[\nu]])$  is given by the formula
\begin{equation}\label{E:uprod}
f _\alpha \ast_\alpha g_\alpha = \left((f_\alpha u_\alpha) \ast v_\alpha \ast (g_\alpha u_\alpha)\right) \tilde u_\alpha.
\end{equation}

\begin{lemma}\label{L:constu}
If the matrix $u_\alpha \in Mat_d(C^\infty(U_\alpha))$ is constant, then $\ast_\alpha= \ast$.
\end{lemma}
\begin{proof}
 If $u_\alpha$ is a constant matrix, then $f_\alpha u_\alpha = f_\alpha \ast u_\alpha$  in (\ref{E:astu})  by the normalization property of a star product. 
It follows that (\ref{E:astu})  is the identity mapping,
\[
    \psi_\alpha = (f_\alpha u_\alpha) \ast v_\alpha = (f_\alpha \ast u_\alpha) \ast v_\alpha  = f_\alpha,
\]
whence the lemma follows.
\end{proof}
\begin{lemma}\label{L:uglob}
  The local products $\ast_\alpha$ on $End(E^\ast|_{U_\alpha})[[\nu]]$ define a global star product on 
$C^\infty(End(E^\ast))[[\nu]]$.
\end{lemma}
\begin{proof}
Given a section $f \in C^\infty(End(E^\ast))[[\nu]]$, consider its trivializations $f_\alpha$ and  $f_\beta$ on $U_\alpha$ and $U_\beta$, respectively. On $U_\alpha \cap U_\beta$ we have
\[
     f_\alpha = a_{\alpha\beta}f_\beta \tilde a_{\beta\alpha}.
\]
Since $a_{\alpha\beta}$ is holomorphic and $b_{\beta\alpha}$ is antiholomorphic, it follows from the separation of variables property of $\ast$ that
\[
    u_\alpha = a_{\alpha\beta} u_\beta  b_{\beta\alpha} =  a_{\alpha\beta} \ast u_\beta \ast  b_{\beta\alpha}.
\]
The pointwise inverses $\tilde a_{\beta\alpha}$ and $\tilde b_{\alpha\beta}$ of $a_{\alpha\beta}$ and $b_{\beta\alpha}$, respectively, are their inverses in the algebra $Mat_d(\A(U_\alpha\cap U_\beta))$ as well. Therefore,
\[
        v_\alpha = \tilde b_{\alpha\beta} \ast v_\beta \ast \tilde a_{\beta\alpha}.
\]
Given global sections $f,g \in C^\infty(End(E^\ast))[[\nu]]$, we have
\begin{eqnarray*}
f_\alpha \ast_\alpha g_\alpha = ((f_\alpha u_\alpha) \ast v_\alpha \ast (g_\alpha u_\alpha))\tilde u_\alpha =\\
((a_{\alpha\beta} f_\beta u_\beta b_{\beta\alpha}) \ast (\tilde b_{\alpha\beta} \ast v_\beta \ast \tilde a_{\beta\alpha}) \ast (a_{\alpha\beta} g_\beta u_\beta b_{\beta\alpha}))(\tilde b_{\alpha\beta} \tilde u_\beta \tilde a_{\beta\alpha})=\\
((a_{\alpha\beta}\ast (f_\beta u_\beta) \ast b_{\beta\alpha}) \ast (\tilde b_{\alpha\beta} \ast v_\beta \ast \tilde a_{\beta\alpha}) \ast (a_{\alpha\beta} \ast (g_\beta u_\beta) \ast b_{\beta\alpha}))\\
(\tilde b_{\alpha\beta} \tilde u_\beta \tilde a_{\beta\alpha})= (a_{\alpha\beta}\ast ((f_\beta u_\beta) \ast v_\beta \ast (g_\beta u_\beta)) \ast b_{\beta\alpha})(\tilde b_{\alpha\beta} \tilde u_\beta \tilde a_{\beta\alpha})=\\
 (a_{\alpha\beta} ((f_\beta u_\beta) \ast v_\beta \ast (g_\beta u_\beta)) b_{\beta\alpha})(\tilde b_{\alpha\beta} \tilde u_\beta \tilde a_{\beta\alpha})=\\
 a_{\alpha\beta} (((f_\beta u_\beta) \ast v_\beta \ast (g_\beta u_\beta)) \tilde u_\beta) \tilde a_{\beta\alpha}=
a_{\alpha\beta} (f_\beta \ast_\beta g_\beta )\tilde a_{\beta\alpha}.
\end{eqnarray*}
It follows that the local sections $f_\alpha \ast_\alpha g_\alpha$ are glued together and yield a global section from  $C^\infty(End(E^\ast))[[\nu]]$ which does not depend on the choice of the cover $\{U_\alpha\}$ and of the trivializations of $E|_{U_\alpha}$. By construction, the resulting global product on $C^\infty(End(E^\ast))[[\nu]]$  is locally isomorphic to the product in the algebra $Mat_d(\A)$.
\end{proof}

We will denote by $\ast_u$ the global star-product with separation of variables on the sections of $End(E^\ast)$ constructed above. 

We call a local section $f$ of $End(E^\ast)$ {\it antiholomorphic} if the section $\tilde ufu$ of the antiholomorphic bundle $End(\bar E)$ is antiholomorphic. Thus, a local section $f$ of $End(E^\ast)$ is holomorphic if $\nabla_{\bar\xi}f=0$ for every vector field $\bar\xi$ of type $(0,1)$ and antiholomorphic if $\nabla_\xi f=0$ for every vector field $\xi$ of type $(1,0)$, respectively.

\begin{lemma}\label{L:usep}
Given local sections $f,g$ of $End(E^\ast)$, if $f$ is holomorphic or $g$ is antiholomorphic, then $f \ast_u g = fg$, where $fg$ is the pointwise composition of $f$ and $g$.
\end{lemma}
\begin{proof}
Fix a holomorphic trivialization of $E|_U$ over an open subset $U \subset M$. Let $f$ be a holomorphic section of $End(E^\ast|_U)$ identified with a matrix from $Mat_d(C^\infty(U))$ with holomorphic entries. It follows from (\ref{E:uprod}) and  (\ref{E:separ}) that
\[
   f \ast_u g =  \left( (f \ast u) \ast v \ast (g u)  \right)\tilde u =
 \left( f \ast (g u)  \right)\tilde u = \left( f  g u  \right) \tilde u = fg.
\]
Now let $g$ be an antiholomorphic section of $End(E^\ast|_U)$ identified  with a matrix from $Mat_d(C^\infty(U))$. Then $\tilde u g u$ is a matrix with antiholomorphic entries. We have  from (\ref{E:uprod}) and  (\ref{E:separ}) that
\begin{eqnarray*}
     f \ast_u g =  \left( (f u) \ast v \ast (u (\tilde u g u))  \right)\tilde u = \left( (f u) \ast v \ast (u \ast (\tilde u g u))  \right)\tilde u = \\ \left( (f u) \ast (\tilde u g u)  \right)\tilde u = \left( (f u)(\tilde u g u)  \right)\tilde u =  fg.
\end{eqnarray*}

\end{proof}
Lemma \ref{L:usep} means that the product $\ast_u$ has the property of separation of variables.

\section{Tensor notations}

Let $(E,u)$ be a holomorphic Hermitian vector bundle of rank $d$ on a K\"ahler-Poisson manifold $M$ of complex dimension~$m$. We fix a coordinate chart $U \subset M$ and a holomorphic trivialization of $E|_U$.  Let $u \in Mat_d(C^\infty(U))$ be the corresponding trivialization of the Hermitian metric on $E|_U$ and $\tilde u$ be its pointwise inverse. 
The Christoffel symbol of the canonical connection $\nabla$ on $E$ corresponding to a holomorphic index $k$ is
\[
       \Gamma_k = \frac{\p u}{\p z^k}\tilde u.           
\]
The Christoffel symbol of $\nabla$ corresponding to an antiholomorphic index is equal to zero.
Given a section $s$ of $C^\infty(E^\ast|_U)$, its covariant derivatives are defined as follows,
\begin{equation}\label{E:nablae}
    \nabla_k s = u \frac{\p}{\p z^k}\left(\tilde u s\right) = \frac{\p s}{\p z^k} - \Gamma_k s \mbox{ and } \nabla_{\bar l}s = \frac{\p s}{\p \bar z^l}.
\end{equation}
For a section $f$ of $End(E^\ast|_U)$,
\[
      \nabla_k f = u \left( \frac{\p}{\p z^k} (\tilde u f u)\right) \tilde u =
 \frac{\p f}{\p z^k} + [f, \Gamma_k]  \mbox{ and } \nabla_{\bar l}f = \frac{\p f}{\p \bar z^l}.
\]
The operators $\nabla_k, 1 \leq k \leq m,$ on the local sections of $E, E^\ast$, and $End(E^\ast)$  pairwise commute for all holomorphic indices and the operators $\nabla_{\bar l}, 1 \leq \bar l \leq m,$ pairwise commute for the antiholomorphic ones. Given a holomorphic tensor index $K = k_1 \ldots k_p$, we set
\[
              \p_K = \frac{\p}{\p z^{k_1}} \ldots \frac{\p}{\p z^{k_p}} \mbox{ and }\nabla_K = \nabla_{k_1} \ldots \nabla_{k_p}.
\]
Similarly we define the operators $\p_{\bar L}$ and  $\nabla_{\bar L}$ for an antiholomorphic tensor index $\bar L = \bar l_1 \ldots \bar l_q$. We will omit the bar in the notation for antiholomorphic indices if it does not lead to a confusion.

For a section $f$ of $End(E^\ast|_U)$,
\begin{equation}\label{E:nablap}
         \nabla_K f =  u \left( \p_K (\tilde u f u)\right) \tilde u \mbox{ and } \nabla_{\bar L}f = \p_{\bar L}f.
\end{equation}

We will work with indexed arrays of $d \times d$-matrices defined at a point of a coordinate chart $U \subset M$, where the indices range from 1 to $m$. These arrays will be referred to simply as (matrix-valued) tensors.
We do not assume that such tensors determine coordinate-invariant geometric objects on ~$U$.

For a tensor index $I = i_1 \ldots i_n$ we set $|I| = n$. Given tensors $f_I$ and $g^I$, for each fixed tensor index $I = i_1 \ldots i_n$ the elements $ f_{i_1 \ldots i_n}$ and $g^{i_1 \ldots i_n}$ are $d \times d$-matrices.
We define the contraction of $f_I$ and $g^I$ by the formulas
\[
                       f_I g^I = \sum_{n=0}^\infty f_{i_1 \ldots i_n} g^{i_1 \ldots i_n} \mbox{ and }  g^I f_I = \sum_{n=0}^\infty g^{i_1 \ldots i_n} f_{i_1 \ldots i_n},
\]
where the summands are matrix products. In particular, the contraction depends on the order of tensor factors. We assume that in these sums only finitely many summands are nonzero.

We introduce a (scalar-valued) tensor $\Delta_K^I$ separately symmetric in $I$ and $K$ such that $\Delta_K^I =0$ if $|I| \neq |K|$ and 
\[
    \Delta_{k_1 \ldots k_n}^{i_1 \ldots i_n} = \frac{1}{n!}\sum_{\sigma \in S_n} \delta_{k_{\sigma(1)}}^{i_1} \ldots \delta_{k_{\sigma(n)}}^{i_n},
\] 
where  $S_n$ is the symmetric group. Given a tensor $f_K$ symmetric in the tensor index $K$, we have
\[
     \Delta_K^I f_I = f_I \Delta_K^I = f_K.
\]
 
Assume that $\ast$ is a star product with separation of variables on $M$ and denote by $\ast_u$ the associated star product with separation of variables on $C^\infty(End(E^\ast))[[\nu]]$.

\begin{lemma}\label{L:tensform}
  The star product $\ast_u$ can be written in local coordinates as follows,
\[
    f \ast_u g =  \left(\nabla_{\bar L} f\right) C^{\bar LK} \left(\nabla_K g\right),
\]
where the (matrix-valued) tensor $C^{\bar LK}$ is separately symmetric in the tensor indices $K$ and $\bar L$. 
\end{lemma}
\begin{proof}
 Using formulas (\ref{E:uprod}) and (\ref{E:nablap}) and the separation of variables property of the product $\ast$ we get
\begin{eqnarray}\label{E:tensa}
   f \ast_u g = ((fu) \ast v \ast (u (\tilde u gu)))\tilde u =\\
 \left(\p_{\bar L}f\right) A^{\bar L K} \left(\p_K (\tilde u g u)\right) \tilde u =
\left(\nabla_{\bar L} f \right) A^{\bar L K} \tilde u \left(\nabla_K g\right) \nonumber
\end{eqnarray}
for some tensor $A^{\bar L K}$. To conclude the proof we take $C^{\bar L K} = A^{\bar L K}\tilde u$.
\end{proof}

{\it Example.} Assume that $\ast$ is the anti-Wick star product on $\C^m$, the vector bundle $E$ on $\C^m$ is trivial, and the matrix $u$ is constant. Then $\nabla_k = \p/\p z^k, \ \nabla_{\bar l} = \p/\p \bar z^l$, and, by Lemma \ref{L:constu},
\begin{equation}\label{E:awick}
    f \ast_u g = f \ast g = \sum_{r=0}^\infty \frac{\nu^r}{r!} g^{\bar l_1 k_1}\ldots g^{\bar l_r k_r}\frac{\p^r f}{\p \bar z^{l_1} \ldots \p \bar z^{l_r}}\frac{\p^r g}{\p z^{k_1} \ldots \p z^{k_r}}.
\end{equation}
The product $\ast_u$ is thus given by the scalar-valued tensor $C^{\bar L K}$ such that $C^{\bar LK}=0$ if $|K| \neq |\bar L|$ and
\[
      C^{\bar l_1 \ldots \bar l_r k_1 \ldots k_r} = \frac{\nu^r}{(r!)^2}\sum_{\sigma \in S_r} g^{\bar l_1 k_{\sigma(1)}} \ldots g^{\bar l_r k_{\sigma(r)}}.
\]

Recall that, as introduced in \cite{GR}, a star product (\ref{E:star}) is called natural if the bidifferential operator $C_r$ is of order not greater than $r$ in each argument. This notion immediately extends to star products on the sections of endomorphism bundles.
 It was proved in \cite{CMP3} that any star product with separation of variables $\ast$ on a K\"ahler-Poisson manifold $M$ is natural, which implies that the product $\ast$ in the matrix algebra $Mat_d(\A)$ is natural. We want to show that the product $\ast_u$ is also natural.

We write the tensor $C^{LK}$ as a series in $\nu$,
\[
    C^{LK} = \sum_{r \geq 0} \nu^r C_r^{LK}.
\]
\begin{proposition}\label{P:nat}
  If the component $C_r^{LK}$ is nonzero, then $|L| \leq r$ and $|K| \leq r$.
\end{proposition}
\begin{proof}
  Assume that in formula~ (\ref{E:tensa}) the matrix-valued functions $f$ and $g$ do not depend on $\nu$ and set $h = h_0 + \nu h_1 + \ldots := v \ast (gu)$. Then
\begin{equation}\label{E:nat}
     ((fu) \ast h)\tilde u = (\nabla_L f) C^{LK} (\nabla_K g). 
\end{equation}
The product $\ast$ in $\Mat_d(\A)$ can be given by formula (\ref{E:star}), where the operators $C_r$ are obtained from the corresponding operators for the product $\ast$ on scalar-valued functions. Equating the coefficients at $\nu^r$ on both sides of (\ref{E:nat}), we get
\begin{equation}\label{E:r}
      \sum_{i+j = r}C_i ((fu), (h_j)) \tilde u =(\nabla_L f) C_r^{LK} (\nabla_K g).
\end{equation}
Since the star product $\ast$ is natural, the order of differentiation of $f$ on the left-hand side of (\ref{E:r}) does not exceed $r$. Therefore, $C_r^{LK}=0$ provided $|L| >r$. It can be proved similarly that $C_r^{LK}=0$ provided $|K| >r$.
\end{proof}
Proposition \ref{P:nat} means that the star product $\ast_u$ is natural.

\section{Formal Calabi functions}\label{S:cal}

In this section we define formal Calabi functions related to a closed (1,1)-form and to a holomorphic Hermitian vector bundle on a complex manifold.

Let $M$ be a complex manifold and $U \subset M$ be a holomorphic coordinate chart. We use the following notations for a smooth function $f(z, \bar z)$ on $U$ and holomorphic formal parameters $\eta^k$ corresponding to the holomorphic coordinates $z^k$,
\[
   f(z+ \eta, \bar z) = \exp\left\{\eta^k \frac{\p}{\p z^k}\right\}f = \sum_K \frac{1}{|K|!} \eta^K \p_K f.
\]
We use similar notations for antiholomorphic formal parameters $\bar\eta^l$ corresponding to $\bar z^l$.

Let $\varkappa$ be a closed (1,1)-form on $M$ and $\Psi$ be a potential of $\varkappa$ on $U$ so that $\varkappa = i \p \bar\p \Psi$ (the potential $\Psi$ exists if $U$ is contractible). We define the formal Calabi function for the form~ $\varkappa$ on $U$,
\[
    D_\varkappa(\eta,\bar\eta) := \Psi(z, \bar z) - \Psi(z+\eta, \bar z) + \Psi(z+\eta,\bar z + \bar \eta) - \Psi(z, \bar z + \bar \eta),
\]
as an element of $(C^\infty(U))[[\eta, \bar\eta]]$. The function $ D_\varkappa(\eta,\bar\eta)$ lies in the ideal generated by the products $\eta^k\bar\eta^l$. We write formally $D_\varkappa(0, \bar\eta)=0$ and $D_\varkappa(\eta,0)=0$. In particular,  $\exp\{D_\varkappa\}$ is a well defined element of $(C^\infty(U))[[\eta, \bar\eta]]$. The function $D_\varkappa(\eta,\bar \eta)$ does not depend on the choice of the potential $\Psi$ of the form $\varkappa$.

Now assume that $E$ is a holomorphic Hermitian vector bundle on $M$  and $u \in Mat_d(C^\infty(U))$ is the Hermitian fiber metric for a fixed trivialization $E|_U \cong U \times \C^d$. We introduce a matrix-valued function
\[
    Q_u(\eta,\bar\eta) := u(z, \bar z)\tilde u(z+\eta, \bar z) u(z+\eta,\bar z+\bar \eta)\tilde u(z, \bar z + \bar \eta)
\]
which we interpret as an operator on $C^\infty(E^\ast|_U)[[\eta,\bar\eta]]$. Since $Q_u$ is the identity operator modulo $\eta, \bar\eta$, the operator
\[
   H_u := \log Q_u
\]
is well defined. We refer to $H_u$ as to the formal Calabi function for $E$. If $E$ is a line bundle, then $H_u = D_\varkappa$ for $\varkappa = i \p \bar \p \log u$, i.e., $H_u$ is the formal Calabi function for the curvature of the canonical connection on $E$.

We introduce the following operators on $C^\infty(E^\ast|_U)[[\eta,\bar\eta]]$,
\[
        x := \eta^k \nabla_k \mbox{ and } y := \bar\eta^l \nabla_{\bar l},
\]
where $\nabla_k$ and $\nabla_{\bar l}$ are as in (\ref{E:nablae}). In particular,
\[
      e^x = u \left(e^{\eta^k \frac{\p}{\p z^k}}\right) \tilde u \mbox{ and } e^y = e^{\bar\eta^l \frac{\p}{\p \bar z^l}}.
\]

To simplify the notations we will drop the subscript $u$ in $Q_u$ and $H_u$.
\begin{lemma}\label{L:commut}
The following formula holds,
\[
     Q = e^x e^y e^{-x} e^{-y}.
\]
\end{lemma}
\begin{proof}
Given $f(z,\bar z) \in (C^\infty(U))^d$, we have
\begin{eqnarray*}
    e^x e^y e^{-x} e^{-y}f =  e^x e^y e^{-x} e^{-y}(f(z,\bar z - \bar \eta)) =\\
  e^x e^y (u(z, \bar z) e^{-\eta^k \frac{\p}{\p z^k}}(\tilde u(z, \bar z) f(z,\bar z - \bar \eta)))=\\
e^x e^y (u(z, \bar z)(\tilde u(z-\eta, \bar z) f(z-\eta,\bar z - \bar \eta)))=\\
e^x (u(z, \bar z + \bar \eta)(\tilde u(z-\eta, \bar z + \bar \eta) f(z-\eta,\bar z)))=Q f.
\end{eqnarray*}
\end{proof}
The operators $x$ and $y$ (topologically) generate a pronilpotent Lie algebra of operators on $C^\infty(E^\ast|_U)[[\eta,\bar\eta]]$ which contains $H = \log Q$. The element $H$ can be calculated using the Dynkin form of the Campbell-Baker-Hausdorff formula,
\begin{eqnarray}\label{E:ducomm}
  H = [x,y] + \frac{1}{2!}[x+y,[x,y]] + \\
\frac{1}{3!}\left(\frac{1}{2}[x,[y,[y,x]]] + [x+y,[x+y,[x,y]]]\right)+ \ldots \nonumber
\end{eqnarray}
Writing $H$ in the tensor form,
\begin{equation}\label{E:gkl}
    H = \sum_{K,\bar L}\frac{1}{|K|! |L|!} H_{K\bar L} \eta^K \bar \eta^L,
\end{equation}
where $H_{K\bar L}=0$ unless $|K| \geq 1$ and $|L| \geq 1$, we can express the tensor components  $H_{K\bar L}$ separately symmetric in $K$ and $\bar L$ in terms of covariant derivatives of the curvature $R_{k\bar l}$ of the canonical connection on $E$ using formula (\ref{E:ducomm}). The homogeneous component of $H$ of bidegree (1,1) with respect to the variables $\eta$ and $\bar \eta$ is
\[
    H_{k\bar l}\eta^k \bar \eta^l  = [x,y] =  [\nabla_k, \nabla_{\bar l}] \eta_k \bar \eta_l=- i R_{k \bar l}\eta_k \bar \eta_l,
\]
whence $H_{k\bar l} =-i R_{k\bar l}$. Observe that if $E$ is a line bundle, then
\[
      H_{K\bar L} =\p_K \p_{\bar L} \log u 
\]
if $|K| \geq 1$ and $|L| \geq 1$ and $H_{K\bar L}=0$ otherwise.

\section{Inversion formulas}\label{S:inv}

Given a pseudo-K\"ahler manifold $(M, \omega_{-1})$ of complex dimension $m$, assume that $\ast$ is a star product with separation of variables parameterized by a formal form $\omega = \frac{1}{\nu}\omega_{-1} + \omega_0 + \ldots$ and $E$ is a holomorphic Hermitian vector bundle of rank~ $d$ on $M$. Let $U \subset M$ be a contractible holomorphic coordinate chart, $\Phi \in C^\infty(U)[\nu^{-1}, \nu]]$ be a formal potential of $\omega$, and $u \in Mat_d(C^\infty(U))$ be the Hermitian fiber metric for a fixed trivialization $E|_U \cong U \times \C^d$. We denote by $D = D_\omega$ the formal Calabi function for the form $\omega$ on $U$,
\begin{equation}\label{E:calscal}
  D(\eta,\bar\eta) := \Phi(z, \bar z) - \Phi(z+\eta, \bar z) + \Phi(z+\eta,\bar z + \bar \eta) - \Phi(z, \bar z + \bar \eta).
\end{equation}
It is an element of $(C^\infty(U)[\nu^{-1},\nu]])[[\eta,\bar\eta]]$. We introduce the matrix-valued function
\begin{eqnarray}\label{E:calabi}
    \E(\eta, \bar\eta) : = e^{D(\eta,\bar\eta)} Q(\eta,\bar\eta) = e^{D+H} = \\
 e^{D} u(z, \bar z)\tilde u(z+\eta, \bar z) u(z+\eta,\bar z+\bar \eta)\tilde u(z, \bar z + \bar \eta),\nonumber
\end{eqnarray}
where $Q = Q_u$ and $H=H_u$ are as in Section \ref{S:cal}. The function $\E(\eta,\bar\eta)$ can be written in the tensor form,
\[
           \E(\eta, \bar\eta) =  \sum_{K, \bar L} \frac{1}{|K|! |\bar L|!} E_{K\bar L}\eta^K \bar \eta^L,
\]
where $E_{K\bar L}$ is separately symmetric in $K$ and $\bar L$. In this section we will prove inversion formulas connecting the tensors $E_{K\bar L}$ and $C^{\bar LK}$.
\begin{theorem}\label{T:inv}
The following formulas hold:
\begin{equation}\label{E:inv}
     E_{K\bar L}C^{\bar LI} = \Delta_K^I \mbox{ and } C^{\bar LI} E_{I\bar J} = \Delta_{\bar J}^{\bar L}.
\end{equation}
\end{theorem}

We break the proof into a series of technical lemmas. Given a matrix~ $a$, we denote by $r_a$ the right pointwise multiplication operator by~ $a$.
\begin{lemma}\label{L:left}
   For $g \in Mat_d(C^\infty(U))$ the following equalities hold:
\begin{equation}\label{E:left}
    \left(\frac{\p\Phi}{\p z^k} + \Gamma_k\right) \ast_u g = \frac{\p g}{\p z^k} +  \frac{\p\Phi}{\p z^k}  g + g \Gamma_k = \left(r_{\tilde u} e^{-\Phi} \frac{\p}{\p z^k} e^\Phi r_u\right) g.
\end{equation}
\end{lemma}
\begin{proof}
Using  formula (\ref{E:pphik}) and taking 
\[
f = \left(\frac{\p\Phi}{\p z^k}\ast u\right)\tilde u =  \frac{\p\Phi}{\p z^k} + \frac{\p u}{\p z^k}\tilde u = 
\frac{\p\Phi}{\p z^k} + \Gamma_k
\]
in (\ref{E:uprod}), we obtain the first equality in (\ref{E:left}):
\begin{eqnarray*}
\left(\frac{\p\Phi}{\p z^k} + \Gamma_k\right) \ast_u g  =
   \left( \left(\frac{\p\Phi}{\p z^k}\ast u\right)\tilde u\right) \ast_u g = \\
\left(\frac{\p\Phi}{\p z^k}\ast u \ast v \ast (gu)\right)\tilde u =
\left(\frac{\p\Phi}{\p z^k} \ast (gu)\right)\tilde u =\\
 \frac{\p\Phi}{\p z^k } g +  g \frac{\p u}{\p z^k} \tilde u +  \frac{\p g}{\p z^k} =
\frac{\p g}{\p z^k} +  \frac{\p\Phi}{\p z^k}  g + g \Gamma_k .
\end{eqnarray*}
The second equality in (\ref{E:left}) is straightforward.
\end{proof}
\begin{lemma}\label{L:leftexp}
  The following formula holds:
\begin{eqnarray}\label{E:leftexp}
\left(e^{\Phi(z+\eta,\bar z) - \Phi(z,\bar z)}u(z+\eta,\bar z) \tilde u(z,\bar z)\right) \ast_u g =\nonumber\\
 e^{\Phi(z+\eta,\bar z) - \Phi(z,\bar z)}g(z+\eta,\bar z) u(z+\eta,\bar z) \tilde u(z,\bar z).
\end{eqnarray}
\end{lemma}
\begin{proof}
Lemma \ref{L:left} means that the operator
\[
    X_k :=  r_{\tilde u} e^{-\Phi}\left( \frac{\p}{\p z^k}\right) e^\Phi r_u
\]
is a left multiplication operator with respect to the star product $\ast_u$. Therefore,
\[
  \exp\left\{\eta^k X_k\right\} = r_{\tilde u} e^{-\Phi}\left( e^{\eta^k\frac{\p}{\p z^k}}\right) e^\Phi r_u
\]
is a left multiplication operator with respect to the product $\ast_u$ extended to the formal series in $\eta,\bar\eta$ by $(\eta,\bar\eta)$-linearity. We compute directly that
\[
   \exp\left\{\eta^k X_k\right\}g = e^{\Phi(z+\eta,\bar z) - \Phi(z,\bar z)}g(z+\eta,\bar z) u(z+\eta,\bar z) \tilde u(z,\bar z)
\]
for a matrix-valued function $g(z,\bar z)$. Applying the operator $\exp\left\{\eta^k X_k\right\}$ to the identity matrix (which is the identity for the product $\ast_u$), we obtain that it is the left $\ast_u$-multiplication operator by
\[
    e^{\Phi(z+\eta,\bar z) - \Phi(z,\bar z)}u(z+\eta,\bar z) \tilde u(z,\bar z),
\]
which concludes the proof.
\end{proof}
\begin{lemma}\label{L:expnabla}
 The following formula holds:
\begin{equation}\label{E:expnabla}
     \exp\{\eta^k \nabla_k\}g = u(z,\bar z) \tilde u(z+\eta,\bar z)g(z+\eta,\bar z)u(z+\eta, \bar z) \tilde u(z,\bar z).
\end{equation}
\end{lemma}
\begin{proof}
The proof is obtained  from the formula
\[
   \exp\{\eta^k \nabla_k\}g  =  u \left(\exp\left\{\eta^k \frac{\p}{\p z^k}\right\}(\tilde u g u)\right)\tilde u
\]
by a direct computation.
\end{proof}
Using Lemma \ref{L:tensform}, we rewrite (\ref{E:leftexp}) as follows,
\begin{eqnarray}\label{E:exptens}
e^{\Phi(z+\eta,\bar z) - \Phi(z,\bar z)}g(z+\eta,\bar z) u(z+\eta,\bar z) \tilde u(z,\bar z) =\\
\nabla_{\bar L}\left(e^{\Phi(z+\eta,\bar z) - \Phi(z,\bar z)}u(z+\eta,\bar z) \tilde u(z,\bar z)\right) C^{\bar L I} \nabla_I g. \nonumber
\end{eqnarray}
We set
\begin{eqnarray}\label{E:ql}
     Y_{\bar L}(z, \bar z, \eta) := u(z,\bar z)\tilde u(z+\eta,\bar z) e^{-\Phi(z+\eta,\bar z) + \Phi(z,\bar z)}\\
\nabla_{\bar L}\left(e^{\Phi(z+\eta,\bar z) - \Phi(z,\bar z)}u(z+\eta,\bar z) \tilde u(z,\bar z)\right).\nonumber
\end{eqnarray}
Formulas  (\ref{E:expnabla}) and (\ref{E:exptens}) imply that
\begin{eqnarray}\label{E:expetanabla}
     \exp\{\eta^k \nabla_k\}g = Y_{\bar L} C^{\bar L I} \nabla_I g.
\end{eqnarray}
Since $g$ is arbitrary, (\ref{E:expetanabla}) is equivalent to the equality
\[
    \frac{1}{|I|!}\eta^I =  Y_{\bar L}C^{\bar L I}.
\]
To finish the proof of the first formula in (\ref{E:inv}) it suffices to verify that 
\[
             Y_{\bar L} = \sum_K \frac{1}{|K|!}\eta^K E_{K\bar L} 
\]
or, equivalently, that
\[
    \sum_{\bar L} \frac{1}{|\bar L|!}  Y_{\bar L} \bar \eta^L  = \sum_{K, \bar L} \frac{1}{|K|! |\bar L|!}E_{K\bar L}\eta^K \bar \eta^L = \E(\eta, \bar\eta).
\]
We have from (\ref{E:ql}) and the fact that $\nabla_{\bar l} = \frac{\p}{\p \bar z^l}$ that
\begin{eqnarray*}
    \sum_{\bar L} \frac{1}{|\bar L|!} Y_{\bar L} \bar \eta^L = 
 u(z,\bar z)\tilde u(z+\eta,\bar z) e^{-\Phi(z+\eta,\bar z) + \Phi(z,\bar z)}\\
\exp\left\{\bar\eta^l\frac{\p}{\p \bar z^l}\right\}\left(e^{\Phi(z+\eta,\bar z) - \Phi(z,\bar z)}u(z+\eta,\bar z) \tilde u(z,\bar z)\right) = \E(\eta, \bar\eta).
\end{eqnarray*}
The second formula can be proved similarly starting with (\ref{E:pphil}). Theorem~ \ref{T:inv} is proved.

Lemma \ref{L:left} has the following important application. Assume that $(E, u)$ is a holomorphic Hermitian line bundle. Since the endomorphism bundle of a line bundle is the trivial line bundle, then the associated product $\ast_u$ is a global star product with separation of variables on the scalar-valued functions on $M$. Denote by $\varkappa$ the curvature $(1,1)$-form of the canonical Hermitian connection on $(E,u)$  given locally by the formula $\varkappa = i \p\bar\p\log u$.

\begin{proposition}\label{P:twisted}
The characterizing form of the star product $\ast_u$ associated to a holomorphic Hermitian line bundle $(E,u)$  is $\omega + \varkappa$. 
\end{proposition}
\begin{proof}
In the scalar case ($d=1$),
\[
     \Gamma_k = \frac{\p \log u}{\p z^k}.
\] 
It follows from (\ref{E:left}) that
\[
   \frac{\p(\Phi + \log u)}{\p z^k} \ast_u g  = \frac{\p(\Phi + \log u)}{\p z^k } g +   \frac{\p g}{\p z^k}.
\]
Now the statement of the proposition follows from the condition (\ref{E:pphik}).
\end{proof}

\section{Operators on a formal Fock space}

In this section we define a formal Fock space $\V$ and an algebra of operators on $\V$. We will use the same assumptions as in Section \ref{S:inv}. We introduce scalar-valued tensors $G_{K\bar L}$ and $G^{\bar LK}$ separately symmetric in $K$ and $\bar L$, such that  $G_{K\bar L}=0$ and $G^{\bar LK}=0$ if $|K| \neq |\bar L|$,
\begin{eqnarray*}
      G_{k_1 \ldots k_r \bar l_1 \ldots \bar l_r} = \frac{1}{\nu^r r!} \sum_{\sigma \in S_r} g_{k_1 \bar l_{\sigma(1)}} \ldots g_{k_r \bar l_{\sigma(r)}}, \mbox{ and }\\
      G^{ \bar l_1 \ldots \bar l_rk_1 \ldots k_r} = \frac{\nu^r}{r!} \sum_{\sigma \in S_r} g^{\bar l_{\sigma(1)}k_1} \ldots g^{\bar l_{\sigma(r)}k_r}.
\end{eqnarray*}
It can be checked that
\[
      G_{K\bar L} G^{\bar L I} = \Delta_K^I \mbox{ and } G^{\bar JK} G_{K\bar L} = \Delta_{\bar L}^{\bar J}.
\]
The tensors  $G_{K\bar L}$ and $G^{\bar LK}$ will be used to raise and lower tensor indices. We define matrix-valued tensors $C_K^I$ and $E_K^I$ as follows:
\begin{eqnarray*}
C_K^I = G_{K\bar L}C^{\bar LI} \mbox{ and } E_K^I = E_{K\bar L} G^{\bar L I}.
\end{eqnarray*}
Formulas (\ref{E:inv}) imply that
\begin{equation}\label{E:invop}
     E_K^P C_P^I = \Delta_K^I  \mbox{ and } C_K^P E_P^I = \Delta_K^I.
\end{equation}
We want to interpret (\ref{E:invop}) as  inversion formulas for operators on a formal Fock space. A formal Fock space $\V$ is defined as the set of all $\nu$-formal matrix-valued tensors
\[
    f = f_I  = \sum_{r \geq 0} \nu^r f_{r, I}
\]
whose components $f_{r, I}$ are symmetric in $I$. We denote by $\O$ the space of $\nu$-formal matrix-valued tensors
\[
     A_K^I =  \sum_{r \geq 0} \nu^r A_{r, K}^I
\]
separately symmetric in $I$ and $K$ and such that for any fixed $p$ and $r$ there is an integer $q$ such that $A_{r, K}^I=0$ provided $|K|=p$ and $|I| > q$. A tensor $A_K^I \in \O$ acts on the formal Fock space $\V$ as a $\nu$-linear operator with the (infinite) matrix $A_K^I$,
\[
    \V \ni f=f_K \mapsto A_K^I f_I.
\] 
The tensors from $\O$ form an algebra with the identity ~$\Delta_K^I$. The composition of tensors $A_K^I$ and $B_K^I$ is $A_K^P B_P^I$.

We want to show that the tensor $C_K^I$ is from $\O$. We write it as a series in $\nu$,
\[
    C_K^I = \sum_{s \in \Z} \nu^s C_{s, K}^I,
\]
whence
\begin{equation}\label{E:cc}
  C_{s, k_1 \ldots k_p}^{i_1 \ldots i_q} = g_{k_1\bar l_1} \ldots g_{k_p\bar l_p} C_{s+p}^{\bar l_1 \ldots \bar l_p i_1 \ldots i_q}.
\end{equation}
The star product $\ast_u$ is natural. Therefore, if the component $C_r^{\bar LK}$ is nonzero, then $|L| \leq r$ and $|K| \leq r$. It follows from (\ref{E:cc}) that if $C_{s, K}^I$ is nonzero, then $s \geq 0$ and $q \leq s + p$, which implies that $C_K^I$ is from $\O$.

We have proved the following statement.
\begin{proposition}\label{P:}
  The tensor $C_K^I$ defines an operator on the formal Fock space $\V$.
\end{proposition}

In the subsequent sections we will show that the tensor $E_K^I$ is also from $\O$. Then (\ref{E:invop}) will imply that the operators on $\V$ corresponding to $E_K^I$ and $C_K^I$ are inverse to each other.

\section{Feynman graphs and tensors}

In this section we use the same assumptions as in Section \ref{S:inv}. We will define a set $\M$ of equivalence classes of Feynman graphs and its subset ~$\Ncal$. Recall that a directed graph without cycles has a natural partial order $\succ$ on its vertices such that for vertices $v$ and $w$ we have $v \succ w$ if there is a directed path from $v$ to $w$. 

A graph from $\M$ is a directed graph with multiple edges and without cycles. It has two external vertices, the source $\i$ with no incoming edges, and the sink $\o$ with no outgoing edges. It has a possibly empty set of internal vertices, which are regular or special. A regular vertex has an integral weight $r \geq -1$. The special vertices are linearly ordered. The linear order $>$ on the special vertices agrees with the partial order $\succ$ on all vertices so that if $s_1 \succ s_2$ for special vertices $s_1,s_2$, then $s_1 > s_2$ (a directed path connecting $s_1$ and $s_2$ may pass through regular vertices). Each internal vertex has at least one incoming and at least one outgoing edge. Moreover, the total number of edges incident to a regular vertex of weight $r = -1$ is at least three.  The type of a regular internal vertex of weight $r$ is the triple $(p,q,r)$  and the type of a special internal vertex is the pair $(p,q)$, where $p$ and $q$ are the numbers of incoming and outgoing edges, respectively. An isomorphism of two graphs from~ $\M$ respects the source, the sink, the regular internal vertices, their weights, the special internal vertices, and their linear order. For a graph $\Gamma$ from $\M$ we denote by $[\Gamma] \in \M$ its equivalence class, by $Aut(\Gamma)$ the group of auto\-morphisms of $\Gamma$, by  $|Aut(\Gamma)|$ its order, by $\mathbf{s}(\Gamma)$ the number of special internal vertices, by $\pp(\Gamma)$ the degree of the sink, and by $\qq(\Gamma)$ the degree of the source of $\Gamma$. The subset $\Ncal \subset \M$ consists of the equivalence classes of graphs which have no edges connecting internal vertices. Each edge of a graph from $\Ncal$ connects an internal vertex with the source or the sink, or connects the source directly with the sink.

We relate to each graph $\Gamma$ from $\M$ matrix-valued tensors $\Gamma_{K\bar L}, \Gamma_K^I$, and $\Gamma^{\bar LK}$ separately symmetric in $I,K,$ and $\bar L$ and such that
\[
     \Gamma_K^I = \Gamma_{K \bar L} G^{\bar L I} = G_{K \bar L} \Gamma^{\bar L I}. 
\]

The tensor $\Gamma^{\bar L K}$ is constructed as follows. To each regular internal vertex of type $(p,q,r)$ we relate the scalar-valued function
\begin{equation}\label{E:scal}
      \nu^r\frac{\p^{p+q}\Phi_r}{\p z^{k_1} \ldots z^{k_p} \bar z^{l_1} \ldots \bar z^{l_q}}.
\end{equation}
To each special  internal vertex of type $(p,q)$ we relate the matrix-valued function
\begin{equation}\label{E:matr}
     H_{k_1 \ldots k_p \bar l_1 \ldots \bar l_q },
\end{equation}
where the tensor $H_{K\bar L}$ is given by formula (\ref{E:gkl}).
To each edge we relate the tensor $\nu g^{\bar lk}$ so that the holomorphic index $k$ corresponds to its head and the antiholomorphic index $\bar l$ corresponds to its tail.  Then we contract the upper and lower indices according to the incidences of edges and vertices and compose the matrix-valued factors corresponding to the special vertices  in the descending order from left to right (i.e., the special vertex corresponding to the left-most matrix-valued factor is the largest with respect to the linear order). Finally, we separately symmetrize the resulting matrix-valued tensor in the holomorphic and antiholomorphic indices corresponding to the sink and the source, respectively. Observe that $\Gamma^{\bar L K} = 0$ unless $|K| = \pp(\Gamma)$ and $|\bar L| = \qq(\Gamma)$.

In order to construct the tensor $\Gamma_K^I$, we relate to each regular internal vertex of type $(p,q,r)$ the scalar-valued function
\[
     \nu^{r+q} \frac{\p^{p+q}\Phi_r}{\p z^{k_1} \ldots z^{k_p} \bar z^{l_1} \ldots \bar z^{l_q}}g^{\bar l_1 i_1}\ldots g^{\bar l_q i_q}
\]
and we relate to each special vertex  of type $(p,q)$ the matrix-valued function
\[
    \nu^q H_{k_1 \ldots k_p \bar l_1 \ldots \bar l_q }g^{\bar l_1 i_1}\ldots g^{\bar l_q i_q}.
\]
The lower and upper holomorphic indices correspond to the incoming and outgoing edges, respectively.
To each edge directly connecting the source with the sink we relate a copy of the Kronecker tensor $\delta_k^i$. Then we contract the upper and lower holomorphic indices according to the incidences of edges and vertices and compose the matrix-valued factors corresponding to the special vertices according to their order.  Finally, we separately symmetrize the resulting matrix-valued tensor in  the upper and lower holomorphic indices corresponding to the sink and the source, respectively. We have that $\Gamma_K^I=0$ unless $|I|= \pp(\Gamma)$ and $|K|= \qq(\Gamma)$.

We will need the tensor $\Gamma_{K\bar L}$ only if the graph $\Gamma$ is from $\Ncal$, in which case we construct it as follows. We relate to each  regular internal vertex of type $(p,q,r)$ the scalar-valued function (\ref{E:scal}) and to each special  internal vertex of type $(p,q)$  the matrix-valued function (\ref{E:matr}). To each edge directly connecting the source with the sink we relate the tensor $(1/\nu)g_{k\bar l}$. We ignore the other edges. Thus, in a graph from $\Ncal$, we treat an edge connecting the source directly with the sink as if it contains an invisible regular internal vertex of type $(1,1,-1)$. Then we multiply the functions corresponding to the internal vertices so that the matrix-valued functions are multiplied according to the order of the corresponding special vertices, and separately symmetrize the resulting tensor with respect to the holomorphic and antiholomorphic indices. We have that $\Gamma_{K\bar L}=0$ unless $|K|=\qq(\Gamma)$  and $|L| = \pp(\Gamma)$.

{\it Example.}  The following graph $\Gamma$ with one regular internal vertex of type $(1,2,3)$ is from $\Ncal$:
\begin{equation}\label{E:graph}
 { \xygraph{
!{<0cm,0cm>;<1cm,0cm>:<0cm,0cm>::}
!{(0,0) }*+{\circ_{\mathbf{in}}}="a"
!{(1,0) }*+{\bullet_{3}}="b"
!{(2.5,0) }*+{\circ_{\mathbf{out}}}="c"
"a":"b"
"b":@/^/"c"
"b":@/_/"c"
} }.
\end{equation}

For this graph $\pp(\Gamma) = 2, \qq(\Gamma)=1$, and $\mathbf{s}(\Gamma)=0$.  The tensor $\Gamma^{\bar L K} =0$ unless $|\bar L| = 1$ and $|K|=2$, in which case
\[
      \Gamma^{\bar l k_1 k_2} = \nu^6 g^{\bar lk}\frac{\p \Phi_3}{\p z^k \p \bar z^{l_1} \p \bar z^{l_2}} g^{\bar l_1 k_1}g^{\bar l_2 k_2}.
\]
The tensor $\Gamma_K^I=0$ unless $|I|= 2$ and $|K|=1$, and
\[
       \Gamma_{k}^{i_1 i_2} =  \nu^5 \frac{\p \Phi_3}{\p z^k \p \bar z^{l_1} \p \bar z^{l_2}} g^{\bar l_1 i_1}g^{\bar l_2 i_2}.
\]
The tensor $\Gamma_{K\bar L}=0$ unless $|K|=1$ and $|\bar L|=2$, and
\[
     \Gamma_{k \bar l_1 \bar l_2} = \nu^3 \frac{\p \Phi_3}{\p z^k \p \bar z^{l_1} \p \bar z^{l_2}}.
\]

To prove the following lemma we slightly modify the proof of Lemma~ 4 from \cite{CMP5}.

\begin{lemma}\label{L:four}
  Let $a(\Gamma)$ be an arbitrary complex-valued function on ~$\M$. Then the tensor
\begin{equation}\label{E:format}
 A_K^I = \sum_{[\Gamma] \in \M} a(\Gamma)\Gamma_K^I
\end{equation}
is from $\O$ and thus determines an operator on the formal Fock space~ $\V$. The summation in (\ref{E:format}) is over a set of representatives of the classes from $\M$.
\end{lemma}

\section{A graph-theoretic formula for the tensor $E_K^I$}

In this section we consider Feynman graphs only from $\Ncal$. We denote by $\N$ the set of positive integers and by $\Z_{\geq 0}$ the set of nonnegative integers, and set $[k] :=\{1,2,\ldots k\}$ for $k \in \N$ and $[0]:=\emptyset$.  It will be convenient to treat an edge in a graph from $\Ncal$ connecting the source directly with the sink as if it contains an invisible regular internal vertex of type $(1,1,-1)$. The set of types of regular internal vertices is
\[
\mathbb{T} := \{(p,q,r) \in \Z^3| p,q \geq 1, r \geq -1\}.
\]
Let $\Gamma$ be a graph from $\Ncal$ with $k \geq 0$ special vertices $s_k >  \ldots > s_1$. The ordering functions $\mathbf{P}, \mathbf{Q}:[k] \to \N$ give the type $(\mathbf{P}(i),\mathbf{Q}(i))$ of the special internal vertex ~$s_i$.  The multiplicity function of regular vertices of $\Gamma$ is a mapping $\n: \mathbb{T} \to \Z_{\geq 0}$ with a finite support. The value $\n(1,1,-1)$ is the number of edges connecting the source directly with the sink. 

A graph $\Gamma$ from $\Ncal$ is completely determined by the data $(\n,k, \mathbf{P},\mathbf{Q})$, where $ \mathbf{P},\mathbf{Q} \in \N^{[k]}$. Its group of automorphisms $Aut(\Gamma)$ permutes for each type $(p,q,r)$ the $\n(p,q,r)$ regular internal vertices of that type and the $\mathbf{n}(1,1,-1)$ edges connecting the source directly with the sink, and separately permutes the incoming and outgoing edges of each internal vertex. Therefore, the order of $Aut(\Gamma)$ is a function of $(\n,k, \mathbf{P},\mathbf{Q})$,
$|Aut(\Gamma)| =  \lambda(\n,k, \mathbf{P},\mathbf{Q})$, where
\[
     \lambda(\n,k, \mathbf{P},\mathbf{Q}) =\left( \prod_{(p,q,r)\in \mathbb{T}} \n(p,q,r)! (p! q!)^{\n(p,q,r)}\right) \prod_{i=1}^k (\mathbf{P}(i)! \mathbf{Q}(i)!).
\]

We will write the Calabi functions $D = D_\omega$ and $H= H_u$ as formal series,
\[
    D = \sum_{(p,q,r) \in \mathbb{T}} \frac{1}{p! q!} D_{p,q,r} \mbox{ and } H = \sum_{p,q \geq 1} \frac{1}{p! q!} H_{p,q},
\]
where 
\[
    D_{p,q,r} = \nu^r\frac{\p^{p+q}\Phi_r}{\p z^{k_1} \ldots z^{k_p} \bar z^{l_1} \ldots \bar z^{l_q}}\eta^{k_1} \ldots \eta^{k_p}\bar \eta^{l_1} \ldots \bar \eta^{l_q}
\]
and
\[
  H_{p,q} = H_{k_1 \ldots k_p \bar l_1 \ldots \bar l_q }\eta^{k_1} \ldots \eta^{k_p}\bar \eta^{l_1} \ldots \bar \eta^{l_q}
\]
are both  homogeneous of bidegree $(p,q)$ with respect to the variables~ $\eta,\bar\eta$. We want to express $\E = \exp\{D+H\}$ in terms of $D_{p,q,r}$ and $H_{p,q}$. On the one hand,
\[
  e^D = \sum_{\n}\prod_{(p,q,r)\in \mathbb{T}} \frac{1}{\n(p,q,r)!} \left(\frac{1}{p! q!} D_{p,q,r}\right)^{\n(p,q,r)},
\]
where the summation is over the multiplicity functions. On the other hand,
\[
    e^H = \sum_{k=0}^\infty \frac{1}{k!}\left(  \sum_{\mathbf{P}, \mathbf{Q} \in \N^{[k]}}  \prod_{i=1}^k\frac{1}{\mathbf{P}(i)! \mathbf{Q}(i)!}\prod_{i=1}^k H_{\mathbf{P}(i),\mathbf{Q}(i)}\right),
\]
where we write the product of the matrix-valued factors as
\[
   \prod_{i=1}^k H_{\mathbf{P}(i),\mathbf{Q}(i)} =  H_{\mathbf{P}(k),\mathbf{Q}(k)} H_{\mathbf{P}(k-1),\mathbf{Q}(k-1)} \ldots H_{\mathbf{P}(1),\mathbf{Q}(1)}.
\]
Therefore,
\begin{equation}\label{E:enspq}
     \E = \sum_{(\n,k, \mathbf{P},\mathbf{Q})} \frac{1}{\lambda(\n,k, \mathbf{P},\mathbf{Q}) k!}\prod_{(p,q,r)\in \mathbb{T}}  \left(D_{p,q,r}\right)^{\n(p,q,r)}\prod_{i=1}^k H_{\mathbf{P}(i),\mathbf{Q}(i)},
\end{equation}
where the summation is over the tuples $(\n,k, \mathbf{P},\mathbf{Q})$ with $ \mathbf{P},\mathbf{Q} \in \N^{[k]}$. Observe that if $\Gamma$ is a graph from $\Ncal$ parameterized by  $(\n,k, \mathbf{P},\mathbf{Q})$, then
\[
       \prod_{(p,q,r)\in \mathbb{T}}  \left(D_{p,q,r}\right)^{\n(p,q,r)}\prod_{i=1}^k H_{\mathbf{P}(i),\mathbf{Q}(i)} = \Gamma_{K\bar L} \eta^K \bar \eta^L.
\]
Formula (\ref{E:enspq}) can be rewritten as follows:
\[
     \E(\eta,\bar\eta) = \sum_{[\Gamma] \in \Ncal} \frac{1}{|Aut(\Gamma)|\, \mathbf{s}(\Gamma)!}\Gamma_{K\bar L} \eta^K \bar \eta^L.
\]
Using the fact that $\Gamma_{K\bar L} = 0$ unless $|K| = \qq(\Gamma)$ and $|\bar L| = \pp(\Gamma)$, we get that
\begin{equation}\label{E:ekl}
      E_{K\bar L} =  \sum_{[\Gamma] \in \Ncal} \frac{\pp(\Gamma)! \qq(\Gamma)!}{|Aut(\Gamma)|\, \mathbf{s}(\Gamma)!}\Gamma_{K\bar L}.
\end{equation}
Lifting the antiholomorphic tensor index $\bar L$ in (\ref{E:ekl}) by the tensor $G^{\bar LI}$, we obtain the following theorem:
\begin{theorem}\label{T:eki}
  The tensor $E_K^I$ is given by the graph-theoretic formula
\begin{equation}\label{E:ekigraph}
      E_K^I =  \sum_{[\Gamma] \in \Ncal} \frac{\pp(\Gamma)! \qq(\Gamma)!}{|Aut(\Gamma)|\, \mathbf{s}(\Gamma)!}\Gamma_K^I.
\end{equation}
\end{theorem}
The tensor $E_K^I$ is thus represented in the format of (\ref{E:format}). Lemma \ref{L:four} implies that this tensor is from $\O$. Using formulas (\ref{E:invop}), we arrive at the following theorem.
\begin{theorem}\label{T:invop}
The tensors $E_K^I$ and $C_K^I$  determine inverse operators on the formal Fock space $\V$.
\end{theorem}

\section{A composition formula}

In this section we define two dual operations on graphs from $\M$, concatenation and partition, and use these operations to prove a composition formula for operators on the formal Fock space $\V$ expressed in terms of graphs from $\M$.

Given a graph $\Gamma$ from $\M$, recall that $\Gamma_K^I=0$ unless $|K| = \qq(\Gamma)$ and $|I| = \pp(\Gamma)$. If $\Gamma_1$ and $\Gamma_2$ are two graphs from $\M$, the tensors $ (\Gamma_1)_K^I$ and $ (\Gamma_2)_K^I$ are composable if $\pp(\Gamma_1) = \qq(\Gamma_2)$. If this condition holds, we will say that the graphs $\Gamma_1$ and $\Gamma_2$ are composable. 

We define an operation of concatenation of composable graphs ~$\Gamma_1$ and $\Gamma_2$  from $\M$. For $i=1,2$ let $I_i$ and $O_i$ be the sets of edges incident to the source and the sink of the graph $\Gamma_i$, respectively. We set
\[
      n : = \pp(\Gamma_1) = \qq(\Gamma_2) = |O_1| = |I_2|.
\]
Let $\tau: O_1 \to I_2$ be a bijection.  The concatenation of graphs $\Gamma_1$ and ~$\Gamma_2$ corresponding to $\tau$ is a graph from $\M$ denoted by $\Gamma_1 \#_\tau \Gamma_2$ and constructed as follows.  The sink from $\Gamma_1$ and the source from $\Gamma_2$ are removed and the loose end of each edge $e \in O_1$ is spliced with the loose end of the egde $\tau(e) \in I_2$.
There is a unique linear order on the special internal vertices of $\Gamma_1 \#_\tau \Gamma_2$ which agrees with the linear order on the special vertices of $\Gamma_1$ and $\Gamma_2$ and is such that each special vertex of $\Gamma_1 \#_\tau \Gamma_2$ inherited from $\Gamma_1$ is greater than every special vertex inhertied from~ $\Gamma_2$. With this order, each concatenation $\Gamma_1 \#_\tau \Gamma_2$ is from ~$\M$. Clearly,
\[
      \qq(\Gamma_1 \#_\tau \Gamma_2) = \qq(\Gamma_1) \mbox{ and } \pp(\Gamma_1 \#_\tau \Gamma_2) = \pp(\Gamma_2).
\]
The composition of the tensors $(\Gamma_1)_K^I$ and  $(\Gamma_2)_K^I$ can be given by the formula
\begin{equation}\label{E:prodgamma}
     (\Gamma_1)_K^P (\Gamma_2)_P^I = \frac{1}{n!}\sum_\tau    (\Gamma_1 \#_\tau \Gamma_2)_K^I,
\end{equation}
where the summation is over the $n!$ bijections $\tau:  O_1 \to I_2$.

Given a graph $\Gamma$ from $\M$, we say that $\pi$ is an admissible partition of $\Gamma$ if the set of internal vertices of $\Gamma$ is partitioned into two sets ~$V_1^\pi$ and ~$V_2^\pi$ such that each special vertex from $V_1^\pi$ is greater than every special vertex from $V_2^\pi$ and every edge connecting $\i \cup V_1^\pi$ with $V_2^\pi \cup \o$ has its tail in $\i \cup V_1^\pi$ and head in $V_2^\pi \cup \o$.

{\it Example.} If $\Gamma_1$ and $\Gamma_2$ are composable graphs from $\M$, then any concatenation $\Gamma_1 \#_\tau \Gamma_2$ has a natural admissible partition, $\pi_\tau$, into the vertices inherited from $\Gamma_1$ and $\Gamma_2$. 

 If a partition $\pi$ of $\Gamma$ is admissible, we construct two graphs $\Gamma_1^\pi$ and $\Gamma_2^\pi$ from $\M$ as follows. The graph $\Gamma_1^\pi$ is obtained from $\Gamma$ by removing the vertices from $V_2^\pi$ and the edges between the vertices in $V_2^\pi \cup \o$, and connecting the loose ends to the sink $\o$.  The graph~ $\Gamma_2^\pi$ is obtained from $\Gamma$ by removing the vertices from $V_1^\pi$ and the edges between the vertices in $\i \cup V_1^\pi$, and connecting the loose ends to the source $\i$. The set of edges in $\Gamma$ connecting vertices from $\i \cup V_1^\pi$ with vertices from $V_2^\pi \cup \o$ naturally bijectively corresponds to the set ~$O_1^\pi$ of vertices in~ $\Gamma_1^\pi$ incident to the sink and to the set of vertices $I_2^\pi$  in $\Gamma_2^\pi$ incident to the source. Thus, there is a natural bijection $\tau^\pi: O_1^\pi \to I_2^\pi$. The graph $\Gamma_1^\pi \#_{\tau^\pi} \Gamma_2^\pi$ is canonically identified with the graph $\Gamma$ and the natural partition of $\Gamma_1^\pi \#_{\tau^\pi} \Gamma_2^\pi$ corresponds to $\pi$ under this identification.

Fix graphs $\Gamma, \Gamma_1$, and $\Gamma_2$ from $M$ such that $\Gamma_1$ and $\Gamma_2$ are composable and $\Gamma$ is isomorphic to some concatenation of $\Gamma_1$ and $\Gamma_2$. 
Denote by ~$\mathrm{T}$ the set of bijections  $\tau:  O_1 \to I_2$ such that $\Gamma_1 \#_\tau\Gamma_2$ is isomorphic to $\Gamma$ and by $\Pi$ the set of admissible partitions $\pi$ of $\Gamma$ such that $\Gamma_1^\pi$ and $\Gamma_2^\pi$ are isomorphic to $\Gamma_1$ and $\Gamma_2$, respectively.
\begin{lemma}\label{L:partconc}
The following identity holds,
\begin{equation}\label{E:rednosum}
   \frac{|\mathrm{T}|}{|Aut(\Gamma_1)||Aut(\Gamma_2)|} = \frac{|\Pi|}{|Aut(\Gamma)|}.
\end{equation}
\end{lemma}
\begin{proof}
Given $\tau \in \mathrm{T}$ and $\pi \in \Pi$, denote by $E(\tau,\pi)$ the set of isomorphisms $\gamma$ of graphs $\Gamma_1 \#_\tau\Gamma_2$ and ~$\Gamma$ such that the natural partition $\pi_\tau$ of $\Gamma_1 \#_\tau\Gamma_2$ corresponds to the partition $\pi$ of $\Gamma$ under this isomorphism. For $i = 1,2$ let $\gamma_i$ be an  isomorphism  of $\Gamma_i$ and $\Gamma_i^\pi$ such that $\gamma_1, \gamma_2$ transfer the bijection $\tau: O_1 \to I_2$ to $\tau^\pi: O_1^\pi \to I_2^\pi$.
Denote by $E(\pi,\tau)$ the set of such pairs $(\gamma_1, \gamma_2)$. There is a natural bijection from $E(\tau,\pi)$ to $E(\pi,\tau)$ ($\gamma$ induces $\gamma_1$ and $\gamma_2$ and vice versa).

Consider a bipartite graph $B$ whose set of vertices  consists of the independent sets $\Pi$ and~ $\mathrm{T}$ and the set of edges connecting  $\pi \in \Pi$ with $\tau \in \mathrm{T}$ is identified with $E(\tau,\pi)$ or $E(\pi,\tau)$. Given a bijection $\tau \in \mathrm{T}$, consider any isomorphism $\gamma$ of  $\Gamma_1 \#_\tau\Gamma_2$ and ~$\Gamma$. Transferring the natural partition $\pi_\tau$ of  $\Gamma_1 \#_\tau\Gamma_2$ to $\Gamma$ via~ $\gamma$ we obtain an admissible partition $\pi \in \Pi$. Then ~$\gamma$ corresponds to an edge connecting the vertices $\pi$ and ~$\tau$. It follows that the degree of each vertex from ~$T$ is $|Aut(\Gamma)|$. 

Given a partition $\pi \in \Pi$, consider for $i=1,2$ an isomorphism $\gamma_i$ of $\Gamma_i^\pi$ and~ $\Gamma_i$. Using the isomorphisms ~$\gamma_1$ and ~$\gamma_2$ we transfer the natural bijection $\tau^\pi: O_1^\pi \to I_2^\pi$ to some bijection $\tau \in \mathrm{T}$. Then the pair $(\gamma_1,\gamma_2)$ corresponds to an edge connecting $\pi$ and $\tau$. Therefore, the degree of each vertex from $\Pi$ is $|Aut(\Gamma_1)||Aut(\Gamma_2)|$. Calculating the number of edges in $B$ as $|T| |Aut(\Gamma)|$ and as $|\Pi||Aut(\Gamma_1)||Aut(\Gamma_2)|$, we obtain the identity ~ (\ref{E:rednosum}).
\end{proof}

Assume that $a(\Gamma)$ and $b(\Gamma)$ are complex-valued functions on $\M$. According to Lemma \ref{L:four}, the tensors
\begin{equation}\label{E:twotens}
      A_K^I = \sum_{[\Gamma]\in \M} \frac{a(\Gamma)}{|Aut(\Gamma)|} \Gamma_K^I \mbox{ and }  B_K^I = \sum_{[\Gamma]\in \M} \frac{b(\Gamma)}{|Aut(\Gamma)|} \Gamma_K^I 
\end{equation}
define operators on the formal Fock space $\V$ and thus have a well defined composition. We will prove the following composition formula.
\begin{theorem}\label{T:}
   The composition of tensors (\ref{E:twotens}) is given by the formula
\begin{equation}\label{E:prodab}
    A_K^P B_P^I = \sum_{[\Gamma]\in \M} \frac{1}{|Aut(\Gamma)|}\sum_\pi \frac{a(\Gamma_1^\pi)b(\Gamma_2^\pi)}{n_\pi !} \Gamma_K^I,
\end{equation}
where the summation in the inner sum is over the admissible partitions ~ $\pi$ of the graph $\Gamma$ and $n_\pi = \pp(\Gamma_1^\pi) = \qq(\Gamma_2^\pi)$.
\end{theorem}
\begin{proof} Assume that $\Gamma_1$ and $\Gamma_2$ are graphs from $\M$ such that $\pp(\Gamma_1) = \qq(\Gamma_2)$ and denote by $n$ the common value of $\pp(\Gamma_1)$ and $ \qq(\Gamma_2)$.  We reduce the proof to the case where $a(\Gamma)$ and 
$b(\Gamma)$ are the characteristic functions of $[\Gamma_1]$ and $[\Gamma_2]$, respectively. Taking into account formula~ (\ref{E:prodgamma}), we get that (\ref{E:prodab}) reduces to
\begin{equation}\label{E:reduced}
     \frac{1}{|Aut(\Gamma_1)||Aut(\Gamma_2)|} \sum_\tau (\Gamma_1 \#_\tau\Gamma_2)_K^I =
  \sum_{[\Gamma] \in \M} \frac{\lambda(\Gamma)}{|Aut(\Gamma)|} \Gamma_K^I,
\end{equation}
where the summation on the left-hand side of (\ref{E:reduced}) is over the bijections  $\tau:  O_1 \to I_2$, and $\lambda(\Gamma)$ is the number of admissible partitions~ $\pi$ of $\Gamma$ such that $\Gamma_1^\pi$ and $\Gamma_2^\pi$ are isomorphic to $\Gamma_1$ and $\Gamma_2$, respectively. Clearly, $\lambda(\Gamma)$ is supported on the graphs isomorphic to concatenations of $\Gamma_1$ and~ $\Gamma_2$. Then for each representative $\Gamma$ from $\M$ in the sum on the right-hand side of  (\ref{E:reduced}) the contribution of the tensor $\Gamma_K^I$ to both sides of  (\ref{E:reduced})  is the same according to the identity (\ref{E:rednosum}).
\end{proof}

\section{A formula for the star product $\ast_u$}

In this section we give an expression for the tensor $C^{\bar LK}$ in terms of graphs from $\M$.

We will call an internal vertex of a graph from $\M$ {\it frontal} if all incoming edges incident to this vertex are outgoing only from the source. Fix a graph $\Gamma$ from $M$ and denote by $X(\Gamma)$ the set of its frontal regular vertices.  Assume that $\Gamma$ has $k = \mathbf{s}(\Gamma) \geq 0$ special internal vertices. It will be convenient to enumerate the special vertices in $\Gamma$ in the ascending order, 
$s_k > s_{k-1} >\ldots > s_1$. We denote by $l(\Gamma)$ the length of the longest chain $s_k > s_{k-1} >\ldots$ consisting only of frontal special vertices, i.e., $l(\Gamma) = 0$ if there are no frontal special vertices, the vertices $s_i$ with $k \geq i \geq k - l(\Gamma) +1$ are frontal, and if $l(\Gamma) < k$, then $s_{k - l(\Gamma)}$ is not frontal. 

\begin{lemma}\label{L:frontal}
If $\Gamma$ has at least one internal vertex, then there exists a frontal vertex and $|X(\Gamma)|+l(\Gamma) >0$.
\end{lemma}
\begin{proof}
 For every non-frontal internal vertex there is an incoming edge outgoing from another internal vertex, and following these edges (from the head towards the tail) one will arrive at a frontal vertex.  If $|X(\Gamma)| >0$, we are done. If $|X(\Gamma)|=0$, then all frontal vertices of $\Gamma$ are special. Assume that for $k = \mathbf{s}(\Gamma)$ the vertex ~$s_k$ is not frontal. Then, as shown above, there exists a directed path from a frontal vertex $s_i$ to $s_k$, where $i <k$. This contradicts the assumption that ~$\Gamma$ is from~ $\M$. Therefore~ $s_k$ is frontal and $l(\Gamma) \geq 1$.
\end{proof}
\begin{lemma}\label{L:subgraph}
Let $\Gamma$ be a graph from $M$ with $k \geq 0$ special internal vertices. Then for each $i$ satisfying $0 \leq i \leq k$ there exists a unique admissible partition $\sigma_i$ of the graph $\Gamma$ such that the graph $\Gamma_2^{\sigma_i}$ has no regular frontal vertices and exactly $i$ special vertices.
\end{lemma}
\begin{proof}
The partition $\sigma_i$ of $\Gamma$ is constructed as follows. Let $V$ be the set of all internal vertices and $s_k > s_{k-1} >\ldots > s_1$ be the special vertices of $\Gamma$. Define $V_2^{\sigma_i}$ as the subset of $V$ consisting of the special vertices $s_i > s_{i-1} > \ldots > s_1$ and the regular vertices which can be reached by a directed path starting at one of these special vertices. Observe that each edge whose tail is in $V_2^{\sigma_i}$ can have its head in $V_2^{\sigma_i}$ or at the sink. Set $V_1^{\sigma_i} = V \setminus V_2^{\sigma_i}$. It is easy to verify that the partition $\sigma_i$ satisfies the conditions of the lemma.
\end{proof}

For $n \in \Z_{\geq 0}$ denote by $\Lambda_n$ a graph from $\M$ with no internal vertices and $n$ edges. The tensor $\Delta_K^I$ admits the following representation,
\[
      \Delta_K^I = \sum_{n=0}^\infty (\Lambda_n)_K^I = \sum_{[\Gamma] \in \M} \frac{d(\Gamma)}{|Aut(\Gamma)|}\Gamma_K^I,
\]
where $d(\Lambda_n) = n!$ and $d(\Gamma)=0$ if $\Gamma$ has at least one internal vertex. 

We want to find a function $c(\Gamma)$ on $\M$ such that
\begin{equation}\label{E:tensc}
       C_K^I = \sum_{[\Gamma]\in \M} \frac{c(\Gamma)}{|Aut(\Gamma)|} \Gamma_K^I.
\end{equation}
To this end we have to satisfy the condition that  the tensor given by the right-hand side of (\ref{E:tensc}) is, say, right inverse to $E_K^I$. According to  formulas (\ref{E:ekigraph}) and ~(\ref{E:prodab}) the function $c(\Gamma)$ has to satisfy the equation
\begin{equation}\label{E:dgamma}
      d(\Gamma) = \qq(\Gamma)!   \sum_\pi \frac{c(\Gamma_2^\pi)}{\mathbf{s}(\Gamma_1^\pi)!}
\end{equation}
for every graph $\Gamma$ from $\M$. The summation in (\ref{E:dgamma}) is over the admissible partitions $\pi$ of the graph~ $\Gamma$ such that~ $\Gamma_1^\pi$ is from~ $\Ncal$. For $\Gamma=\Lambda_n$ there is only the trivial partition ~$\pi$. For this partition $\Gamma_1^\pi = \Gamma_2^\pi =\Lambda_n$  and (\ref{E:dgamma}) holds if $c(\Lambda_n)=1$.

Now assume that ~$\Gamma$ has at least one internal vertex and $k \geq 0$ special vertices $s_k > s_{k-1} > \ldots > s_1$. We have to determine the function $c(\Gamma)$ such that the right-hand side of (\ref{E:dgamma}) be zero. We will be looking for the function $c(\Gamma)$ using the following ansatz. According to Lemma \ref{L:subgraph}, there is an admissible partition $\sigma_k$ of $\Gamma$ such that the graph $\tilde \Gamma := \Gamma_2^{\sigma_k}$ has no regular frontal vertices and $k$ special vertices, i.e., $\mathbf{s}(\Gamma) = \mathbf{s}(\tilde\Gamma)$. We will assume that
\[
     c(\Gamma) = (-1)^{R(\Gamma) - R(\tilde\Gamma)} c(\tilde\Gamma),
\]
where $R(\Gamma)$ denotes the number of regular internal vertices of the graph~ $\Gamma$. We will show that under this assumption the function $c(\Gamma)$ is uniquely determined on the graphs with no regular frontal vertices, and therefore on all graphs from $\M$.

An admissible partition~ $\pi$ of $\Gamma$ such that $\Gamma^\pi_1$ is from $\Ncal$ is determined by an arbitrary subset $Y \subset X(\Gamma)$ and an integer $j$ satisfying $k \geq j \geq k - l(\Gamma)$, so that the set of internal vertices of $\Gamma$ inherited by $\Gamma_1^\pi$ is $Y \cup \{s_{j+1}, s_{j+2}, \ldots, s_k\}$ and the special vertices of $\Gamma_2^\pi$ are $s_j > s_{j-1} > \ldots > s_1$. Observe that if $\pi$ is an admissible partition of $\Gamma$ such that $\Gamma_2^\pi$ has $t$ special vertices, then $\widetilde {\Gamma_2^\pi} \cong \Gamma_2^{\sigma_t}$.  We have, using the ansatz,
\begin{eqnarray}\label{E:ansatz}
    \sum_\pi \frac{c(\Gamma_2^\pi)}{\mathbf{s}(\Gamma_1^\pi)!} = \sum_\pi \frac{(-1)^{R(\Gamma_2^\pi) - R(\widetilde{\Gamma_2^\pi})} c(\widetilde{\Gamma_2^\pi})}{\mathbf{s}(\Gamma_1^\pi)!} =\\
\sum_{j = k - l(\Gamma)}^k \frac{c(\Gamma_2^{\sigma_j})}{(k-j)!} \sum_{Y \subset X(\Gamma)} (-1)^{R(\Gamma) - |Y| - R(\Gamma_2^{\sigma_j})}. \nonumber
\end{eqnarray}
 If $X(\Gamma)$ is nonempty, then
\[
     \sum_{Y \subset X(\Gamma)} (-1)^{|Y|} =0.
\]
Therefore, the value of the expression (\ref{E:ansatz}) is zero, which agrees with the condition (\ref{E:dgamma}). Now assume that $X(\Gamma)$ is empty, i.e., the graph $\Gamma$ has no regular frontal vertices. Then the following equation has to be satisfied,
\begin{equation}\label{E:firsteqn}
     \sum_{j = k - l(\Gamma)}^k \frac{  (-1)^{R(\Gamma)  - R(\Gamma_2^{\sigma_j})} c(\Gamma_2^{\sigma_j})}{(k-j)!} = 0.
\end{equation}
For $i$ satisfying $0 \leq i \leq k$ we set $\Gamma_i := \Gamma_2^{\sigma_i}$. In particular, $\Gamma_k = \Gamma, \Gamma_0 = \Lambda_k$, and $\mathbf{s}(\Gamma_i) = i$.  Starting with the graph $\Gamma_i$ in (\ref{E:firsteqn}) instead of $\Gamma$  for $1 \leq i \leq k$, we obtain a triangular system of $k+1$ equations in the variables $c(\Gamma_i), 0 \leq i \leq k,$ whose matrix is unipotent. The first $k$ equations are
\begin{equation}\label{E:triangsyst}
     \sum_{j = i - l(\Gamma_i)}^i \frac{  (-1)^{R(\Gamma_i)  - R(\Gamma_j)} c(\Gamma_j)}{(i-j)!} = 0
\end{equation}
for $i = k, k-1, \ldots, 1$, and the last one is $c(\Gamma_0) =1$. It can be solved explicitly providing the value of $c(\Gamma)$ for a graph $\Gamma$ with no frontal regular vertices, whence one can obtain the function $c(\Gamma)$ on the whole set $\M$.

\begin{theorem}\label{T:cofgamma}
Let $\Gamma$ be a graph from $\M$ with $k$ special vertices $s_k > s_{k-1} > \ldots > s_1$ and $\sigma_i, 0 \leq i \leq k,$ be the admissible partition of $\Gamma$ such that the graph $\Gamma_i = \Gamma_2^{\sigma_i}$ has $i$ special vertices $s_i > s_{i-1} > \ldots > s_1$ and no frontal regular vertices. Let $l_i$ be the length of the longest chain $s_i > s_{i-1} > \ldots$ of frontal special vertices in the graph $\Gamma_i$.  If $\Gamma$ has no special vertices, then
\begin{equation}\label{E:gammel}
      c(\Gamma) = (-1)^{R(\Gamma)}.
\end{equation}
If $k \geq 1$, then
\begin{equation}\label{E:cofgamma}
     c(\Gamma) =  \sum_{n=0}^{k-1}  \sum_{k_0, \ldots, k_{n+1}} \frac{(-1)^{R(\Gamma)+n+1}}{(k_0-k_1)! (k_1-k_2)! \ldots (k_n - k_{n+1})!},
\end{equation}
where the summation in the inner sum is over the $(n+2)$-tuples of integers $\{k_0, \ldots, k_{n+1}\}$ such that $k_0=k, \  k_{n+1}=0$, and $l_{k_i}  \geq k_i - k_{i+1} \geq ~1$ for $0 \leq i \leq n$.
\end{theorem}

{\it Example.} If $\Gamma$ is a graph from $\Ncal$ with no regular vertices and $k$ special vertices, then so are the graphs $\Gamma_i$ and $l(\Gamma_i)=i$. It is easy to check that $c(\Gamma_i) = \frac{(-1)^i}{i!}$. In particular, $\Gamma = \Gamma_k$ and $c(\Gamma) = \frac{(-1)^k}{k!}$.

We conclude that the star product $\ast_u$ is given by the formula
\begin{equation}\label{E:prodstaru}
    f \ast_u g =  \left(\nabla_{\bar L} f\right) \left(\sum_{[\Gamma]\in \M} \frac{c(\Gamma)}{|Aut(\Gamma)|} \Gamma^{\bar LK} \right) \left(\nabla_K g\right).
\end{equation}

{\it Remark.} If, in the notations of Section \ref{S:inv},  $(E,u)$ is the trivial line bundle with $u=1$, then $\ast_u =~ \ast$, the graphs in $\M$ have no special vertices, $c(\Gamma)$ is given by formula~ (\ref{E:gammel}), and (\ref{E:prodstaru}) renders Gammelgaard's formula  for the star product ~$\ast$ with a different sign convention (in the original Gammelgaard's formula $c(\Gamma)=1$) and with the formal parameter $\nu$ incorporated in the tensor $\Gamma^{\bar LK}$.

\end{document}